\documentclass[11pt]{article}
\usepackage{geometry}                
\usepackage{graphicx}
\usepackage{amssymb}
\usepackage{epstopdf}
\title{Two geometric lemmas for $\s$-valued maps\\ and an application to the homogenization of spin systems}
\author{Andrea Braides and Valerio Vallocchia\\
Dipartimento di Matematica
 Universit\`a di Roma ``Tor Vergata''\\
 via della ricerca scientifica 1, 00133 Rome, Italy}
\date{}                                           

\usepackage[english]{babel}
\usepackage[latin1,utf8]{inputenc}
\usepackage[T1]{fontenc}
\usepackage{amsfonts}
\usepackage{amsmath}
\usepackage{amsthm}
\usepackage{amssymb}
\usepackage{indentfirst}
\usepackage{color}
\usepackage{graphicx}
\usepackage{newlfont}
\usepackage{mathrsfs}
\usepackage{latexsym}
\usepackage{wrapfig}
\usepackage{subfig}
\usepackage{braket}
\usepackage{hyperref}
\usepackage{pdfsync}
\usepackage{color}
\usepackage{caption}
\usepackage{pgfplots}
\usepackage{bm}
\usepackage{epigraph}
\usepackage{nomencl}

\usepackage{etoolbox}
\renewcommand\nomgroup[1]{%
  \item[\bfseries
  \ifstrequal{#1}{N}{Basic notation}{%
  \ifstrequal{#1}{P}{Functions space}{%
  \ifstrequal{#1}{O}{Measure theory}{}}}%
]}

\hypersetup{colorlinks=true,linkcolor=blue}

\theoremstyle{plain}
\newtheorem{thm}{Theorem}

\newtheorem{lem}{Lemma}
\newtheorem{prop}{Proposition}

\theoremstyle{definition}

\theoremstyle{remark}

\theoremstyle{remark}
\newtheorem{oss}{\textbf{Remark}}

\newcommand{\R}{\mathbb{R}}
\newcommand{\Z}{\mathbb{Z}}
\newcommand{\B}{{B}^{N}_1}
\def\S{\mathcal{S}^{N-1}}
\def\s{\mathcal{S}^{N-1}}
\newcommand{\N}{\mathbb{N}}
\def\e{\varepsilon}

\def\<{\langle}
\def\>{\rangle}

\makenomenclature

\begin{document}

\maketitle

\begin{abstract}
We prove two geometric lemmas for $\s$-valued functions that allow to modify sequences of lattice spin functions on a small percentage of nodes during a discrete-to-continuum process so as to have a fixed average. This is used to simplify known formulas for the homogenization of spin systems.

\end{abstract}

\section{Introduction}
A motivation for the analysis in the present work is in the study of molecular models where particles are interacting through a potential including both orientation and position variables. In particular we have in mind potentials of Gay-Berne type in models of Liquid Crystals \cite{GAYB,ZAN,deG,Vir,MOLMOD}. In that context a molecule of a liquid crystal is thought of as an ellipsoid with a preferred axis, whose position is identified with a vector $w\in\R^3$ and whose orientation is a vector $u\in\mathcal{S}^2$. Given $\alpha$ and $\beta$ two such particles, the interaction energy will depend  on their orientations $u_\alpha$, $u_\beta$ and the distance vector $\zeta_{\alpha\beta}=w_\beta-w_\alpha$. We will concentrate on some properties on the dependence of the energy on $u$ due to the geometry of $\mathcal{S}^2$ (more in general, of $\s$). 

We restrict to a lattice model where all particles are considered as occupying the sites of a regular (cubic) lattice in the reference configuration. Note that in this assumption $\zeta_{\alpha\beta}=\beta-\alpha$ can be considered as an additional parameter and not a variable. Otherwise, in general the dependence on  $\zeta_{\alpha\beta}$ is thought to be of Lennard-Jones type (for the treatment of such energies, still widely incomplete, we refer to \cite{BLO,BH,BOx}). 

We introduce an energy density $G:\Z^m\times \Z^m\times \s\times \s\to \R\cup \{+\infty\}$, so that
$$G^\xi\left(\alpha,u,v\right)=G(\alpha,\alpha+\xi,u,v)$$
represents the free energy of two molecules oriented as $u$ and $v$, occupying the sites $\alpha$ and $\beta=\alpha+\xi$ in the reference lattice. Note that
we have included a dependence on $\alpha$ to allow for a microstructure at the lattice level, but the energy density is meaningful also in the homogeneous case, with $G^\xi$ independent of $\alpha$. Such energies are the basis for the variational analysis of complex multi-scale behaviours of spin systems (see, e.g., \cite{BRACIC} for the derivation of energies for liquid crystals, \cite{XY} for a study of the XY-model, \cite{CO} for a very refined study of the N-clock model, \cite{CS,CRS} for chirality effects). 

In order to understand the collective behaviour of a spin system, we introduce a small scaling parameter $\e>0$, so that the description of such a behaviour can be formalized as a limit as $\e\to0$. For each Lipschitz set $\Omega$ the discrete set $\Z_\e(\Omega):=\{\alpha\in\e\Z^m: (\alpha+[0,\e)^m)\cap\Omega\neq\emptyset\}$ represents a `discretization' of the set $\Omega$ at scale $\e$. We also let  $R>0$ define a cut-off parameter representing the relevant range of the interactions (which we assume to be finite).

We define the family of scaled  functionals
\begin{align*}
   E_\e(u)=
    \sum_{\substack{\xi\in \Z^m \\ |\xi|\leq R  }}\sum_{\alpha\in R_\e^{\xi}(\Omega)}\e^mG^\xi\left(\frac{\alpha}{\e}, u(\alpha),u(\alpha+\e\xi)\right)
    \end{align*}
with domain functions $u:\Z_\e(\Omega)\to\s$, where $
R_{\e}^{\xi}(\Omega):=\{\alpha \in \Z_\e(\Omega)\ :\ \alpha,\, \alpha+\e \xi \in \Omega \}$.

Extending functions defined on $\Z_\e(\Omega)$ to piecewise-constant interpolations, we may define a discrete-to-continuum convergence of $u_\e$ to $u$. The assumptions on $G$ ensure that $u$ takes values in the unit ball. 
We can then perform an asymptotic analysis using the notation of $\Gamma$-convergence (see e.g.~\cite{BRA1,BH,BT}). Energies as $E_\e$, but with $u$ taking values in general compact sets $K$ have been previously studied by Alicandro, Cicalese and Gloria (2008) \cite{ACG}, who describe the limit with a two-scale homogenization formula.  In the case of $\s$-valued functions we simplify the homogenization formula reducing to test functions $u$ satisfying a constraint on the average. This is a non-trivial fact since this constraint is non-convex, and its proof is the main technical point of the work. 

The key observation is that we can modify the sequences $u_{\e}$ so that they satisfy an exact condition on their average. We formalize this fact in two geometrical lemmas. The first one is a simple observation that each point in the unit ball in $\R^N$ with $N>1$ can be written exactly as the average of $k$ vectors in $S^{N-1}$ for all $k\ge 2$, while the second one allows to modify sequences $u_{\e_j}$ satisfying an asymptotic condition on the discrete average of $u_{\e}$ with a sequence $\widetilde u_\e$ satisfying a sharp one and with the same energy $E_\e$ up to a negligible error. This can be done if the asymptotic average of $u_{\e}$ has modulus strictly less than one. In this case, most of the values of $u_\e$ are not aligned; this  allows to use a small percentage of these values to correct the asymptotic average to a sharp one by using the first lemma.   

Optimizing on all the functions satisfying the same average condition satisfied by their limit we show that
the $\Gamma$-limit of the sequence $E_{\e}$,  for functions $u\in L^{\infty}(\Omega,\B)$ is a continuum functional
$$E_0(u)=\int_{\Omega}G_{\hom}(u)\,dx,$$
and the function $G_{\hom}$ satisfies a homogenization formula
 $$G_{\hom}(z)=
\lim_{T\to\infty}\frac{1}{T^m}\inf\Bigl\{\mathcal{E}_T(u) :\\
 {1\over T^m} \sum_{\alpha\in Z_1(Q_T)}u(\alpha)= z \Bigr\},$$
where $Q_T=(0,T)^m$ and 
$$\mathcal{E}_T(u)=\sum_{\substack{\xi\in \Z^m \\ |\xi|\leq R  }}\sum_{\beta\in R_1^{\xi}(Q_T)} G^\xi\left(\beta, u(\beta), u(\beta+\xi)\right) .$$
Note that the constraint in the homogenization formula involves the values of $u(\alpha)$, which belong to the non-convex set $\s$.
This is an improvement with respect to Theorem $5.3$  in \cite{ACG}, where  the integrand of the limit is characterized imposing a weaker constraint on the average of  $u$; namely it is shown that it equals
\begin{equation}
\label{ABGhom}
\overline G_{\hom}(z)=\lim_{\eta\to 0^+}
\lim_{T\to\infty}\frac{1}{T^m}\inf\Bigl\{
\mathcal{E}_T(u)
:\Bigl| {1\over T^m} \sum_{\alpha\in Z_1(Q_T)}u(\alpha)-z\Bigr|\le \eta  \Bigr\}.
\end{equation}
A formula with a sharp constraint may be useful in higher-order developments, which characterize microstructure, interfaces and singularities. 

\smallskip
The plan of the paper is as follows. In Section \ref{setting} we introduce the notation for discrete-to-continuum homogenization. In Section \ref{lemmas} we state and prove the geometric lemmas on $\s$-valued functions.
In Section \ref{homer} we prove the homogenization formula, and finally in Section \ref{homproof} we give a proof of the homogenization theorem.

 \section{Notation and setting}\label{setting}
Let $m,n\geq 1$, $N\geq 2$ be fixed. We denote by $\{e_1,\dots, e_m\}$ the standard basis of $\R^m$. Given two vectors $v_1,v_2\in \R^n$, by $(v_1, v_2)$ we  denote their scalar product. If $v\in \R^m$, we  use  $|v|$ for the usual euclidean norm. $\S$ is the standard unit sphere of $\R^N$ and $B^N_1$ the closed unit ball of $\R^N$. 
If $x\in \R$, its  integer part  is denoted by $\lfloor x\rfloor$. We also set $Q_T=(0,T)^m$ and $\mathcal{B}(\Omega)$ as the family of all open subsets of $\Omega$.
If $A$ is an open bounded set, given a function $u:A\to \R^N$ we denote its average over $A$ as $$\<u\>_{A}=\frac{1}{|A|}\int_{A}u(x)\,dx.$$

\subsection{Discrete functions}
\label{dsc}
 Let $\Omega\subset \R^m$ be an open bounded domain with Lipschitz boundary, and   let $\e>0$ be the spacing parameter of the cubic lattice $\e\Z^m$. We define the set $$\Z_\e(\Omega):=\{\alpha\in\e\Z^m: \, (\alpha+[0,\e)^m)\cap\Omega\neq \emptyset\}$$ and we will consider discrete functions  $u:\Z_\e(\Omega)\to\s$ defined on the lattice. 
For $\xi \in \Z^m$, we   define
\begin{gather*}
R_{\e}^{\xi}(\Omega):=\{\alpha \in \Z_\e(\Omega)\ :\ \alpha,\, \alpha+\e \xi \in \Omega \},\end{gather*} 
while the ``discrete'' average of a function $v:\Z_\e(A)\to \s$ over an open bounded domain $A$ will be denoted by
\begin{gather*} \<v\>^{d,\e}_A=\frac{1}{\#(\Z_\e(A))}\sum_{\alpha\in\Z_\e(A)}v(\alpha).
\end{gather*}

\subsection{Discrete energies}
We assume that the Borel function $G:\R^m\times \R^m\times \s\times \s\to \R$  satisfies the following conditions
\begin{eqnarray}
\label{fgrowth}
&&\hbox{(boundedness) } \sup\left\{|G(\alpha,\beta,u,v)| : \alpha,\beta\in\R^m, \, u,v\in\s \right\}<\infty;
\\
&&\hbox{(periodicity) there exists $l\in\N$ such that }
 \label{fperiodicity}
 G(\cdot,\cdot,u,v) \text { is }\ Q_l \ \text{ periodic};
 \\
&&\hbox{(lower semicontinuity) }
 \label{fsci}
 G\text{ is lower semicontinuous.}
 \end{eqnarray}

Given $\xi \in \R^m$, we use the notation
\begin{equation}\label{gxi}
G^{\xi}\left(\alpha,u,v\right)=G(\alpha,\alpha+\e\xi,u,v),
\end{equation}
and define the functionals
\begin{equation}\label{discreteenergy2d}
\sum_{\substack{\xi\in \Z^m \\ |\xi|\leq R  }}\sum_{\alpha\in R_\e^{\xi}(\Omega)}\e^mG^\xi\left(\frac{\alpha}{\e}, u(\alpha),u(\alpha+\e\xi)\right)
    \end{equation}
for $u: \Z_\e(\Omega)\to\s$.

\subsection{Discrete-to-continuum convergence}
In what follows we  identify each discrete function $u$  with its piecewise-constant extension $\tilde{u}$ defined by $\tilde{u}(t)=u(\alpha)$ if $t\in \alpha+[0,\e)^m$. We introduce  the sets: 
 \begin{gather*}
 \mathcal{A}_{\e}(\Omega;\s):=\Bigl\{ \tilde{u}:\R^m\to\s \ : \ \tilde{u}(t)\equiv u(\alpha)\  \mbox{if} \ t\in \alpha+[0,\e)^m, \text{for } \alpha\in \Z_\e(\Omega)\Bigr\}.
\end{gather*}
If no confusion is possible, we will simply write $u$ instead of $\tilde{u}$. If $\e=1$ we will simply write $\mathcal{A}(\Omega;\s)$ in the place of $\mathcal{A}_1(\Omega;\s)$. 

Up to the identification of each function $u$ with its piecewise-constant extension, we can consider  energies $E_\e:L^{\infty}(\Omega,\s)\to\R\cup \{+\infty\}$ of the following form:
\begin{equation}
 \label{energy2d}
    E_\e(u;\Omega)=
    \begin{cases}\displaystyle\sum_{\substack{\xi\in \Z^m \\ |\xi|\leq R  }}\sum_{\alpha\in R_\e^{\xi}(\Omega)}\e^mG^\xi\left(\frac{\alpha}{\e}, u(\alpha),u(\alpha+\e\xi)\right)\quad &\text{if } u\in \mathcal{A}_\e(\Omega;\s), \\[3ex]
    +\infty \quad &\text{otherwise}.
    \end{cases}
    \end{equation}

Let $\e_j\to 0$ and let  $\{u_j\}$ be  a sequence of functions $u_j:\Z_{\e_j}(\Omega)\to\s$. We will say that $\{u_j\}$ {\em converges to a function $u$} if $\tilde{u}_j$ is converging to $u$ weakly$^*$ in $L^{\infty}$. Then we will say that the functionals defined in \eqref{discreteenergy2d} $\Gamma$-converge to $E_0$ if $E_\e$ defined in \eqref{energy2d}  $\Gamma$-converge to $E_0$ with respect to that convergence.

\subsection{The homogenization theorem}
We will prove the following discrete-to-continuum homogenization theorem.
\begin{thm}
 \label{mainresult}
 Let $E_{\e}$ be the energy defined in \eqref{discreteenergy2d} and suppose that \eqref{fgrowth}--\eqref{fsci} hold. Then  $E_\e$ $\Gamma$-converge to the functional
\begin{align}\label{home}
E_0(u)=
\displaystyle\int_{\Omega}G_{\hom}(u)\,dx
\end{align}
defined for functions $u\in L^{\infty}(\Omega, B^N_1)$.
The function  $G_{\hom}$ is given by the following asymptotic formula
\begin{align}
\label{fghom}
G_{\hom}(z)=
\lim_{T\to\infty}\frac{1}{T^m}\inf\Bigl\{E_1(u;Q_T):\<u\>^{d,1}_{Q_T}=z  \Bigr\}.
\end{align}
\end{thm}

The treatment of the average condition in \eqref{fghom} will be performed using a geometric lemma which exploits the geometry of $\s$, as shown in the next section.

\section{Two geometric lemmas}\label{lemmas}
In this section we provide two general lemmas. The first one is a simple observation on the characterisation of sums  of vectors in $\S$, while the second one allows to satisfy conditions on the average of discrete functions with values in $\S$. 

     \begin{lem}
\label{lemmapalle}
Let $u$ be  a vector in the ball $B_{k}^N$ in $\R^N$ centred in the origin and with radius $k\geq 2$; then $u$ can be written as the sum of $k$ vectors on $\S$: 
$$
u=\sum_{i=1}^k u_i \qquad u_i\in\S.
$$
 Equivalently, given $u\in B_{1}^N$ and $k\ge 2$, $u$ can be written as the average  of $k$ vectors on $\S$: 
$$
u={1\over k}\sum_{i=1}^k u_i \qquad u_i\in\S.
$$\end{lem}
\begin{proof}
We proceed by induction on $k$.\\
 Let $k=2$ and let $u\in B_2^N$. 
The set $(u+\S)\cap \S$ is not empty set. If we choose $v\in (u+\S)\cap \S$ then the first induction step is proven with $u_1=v$ and $u_2=(u-v)$.

Suppose that the thesis holds for $k-1$. Let $u\in B_k^N$ and note that the set $(u+\S)\cap B_{d-1}^N$ is not empty. If $v\in (u+\S)\cap B_{d-1}^N$ by the inductive hypothesis we may write $v=u_1+\dots+u_{k-1}$ with $u_j\in \S$. The thesis is then proved by setting $u_k=u-v$.
\end{proof}

\begin{lem}
\label{lemmamedia}
Let $A\subset \R^m$ be an open bounded set with Lipschitz boundary. Let $\delta_j>0$ be a spacing parameter and $u_j:\Z_{\delta_j}( A)\to\mathcal{S}^{N-1}$ be a sequence of discrete function. Suppose that $u_j\rightharpoonup^* u$ in $L^\infty (A,B^n_1)$ and that the average of $u$ on $A$
satisfies $|\<u\>_{A}|<1$. Then,  for all $j$ there exist $\tilde{u}_j$ such that 
\begin{enumerate}
\item  the discrete average $\<\tilde{u}_j\>^d_A:=\dfrac{1}{\#\Z_{\delta_j}( A)}\displaystyle\sum_{i\in\Z_{\delta_j}( A)}\tilde{u}_j(i)$ is equal to $\<u\>_{A}$;
\item the function $\tilde{u}_j$ is obtained by modifying the function  $u_j$ in at most $2 P_j$ points, with $\dfrac{P_j}{\#\Z_{\delta_j}( A)}\to 0$.
\end{enumerate} 
\end{lem}

\begin{proof}
To simplify the notation we set $\Z_j(A)=\Z_{\delta_j}( A)$ and  $u_j^i=u_j(i)$.

Note that, by the weak convergence of $u_j$, 
\begin{align}
\label{average} 
\eta_j:=|\<u_j\>^d-\<u\>_{A}|=o(1)
\end{align}
as $j\to+\infty$.
We will treat the case  that $\eta_j\neq 0$ since otherwise we simply take $\tilde u_j=u_j$.

Since $\<u_j\>^d_A\to \<u\>_A$, by the hypothesis that $|\<u\>_A|<1$ we may suppose that
\begin{equation}\label{beq}
|\<u_j\>^d_A|\le 1-2b
\end{equation}
for all $j$, for some $b\in(0,1/2)$.

\smallskip
{\em Claim}: setting $B=b/(4-2b)$, for every $i\in\Z_j( A)$  there exist at least $B\,\#\Z_j( A)$ indices $l\in\Z_j( A)$ such that $(u_j^i, u_j^l)\leq 1-b$.

Indeed, otherwise there exists at least one index $i$ for which  the set 
\begin{align}
\label{contr}
\mathcal{A}_b:=\left\{l\in \Z_j( A) : (u_j^i, u_j^l)>1-b, \ l\neq i\right\}
\end{align} 
is such that $\#\mathcal{A}_b\geq  (1-B)\#\Z_j( A)$
 and  we have 
\begin{eqnarray*}\nonumber
|\<u_j\>^d_A|&\ge&(\<u_j\>^d, u_j^i)=\frac{1}{\#\Z_j( A)}\sum_{l\in \Z_j( A)}(u_j^l, u_j^i)\\
&=&\frac{1}{\#\Z_j( A)}\sum_{l\in \mathcal{A}_b }(u_j^l, u_j^i)+\frac{1}{\#\Z_j( A)}\sum_{l\in \Z_j( A)\backslash\mathcal{A}_b }(u_j^l, u_j^i)\\
&\geq&\frac{1}{\#\Z_j( A)}\left(\# \mathcal{A}_b(1-b)-\left(\#\Z_j( A)-\#\mathcal{A}_b\right)\right)
\\
&=&\frac{1}{\#\Z_j( A)}\Bigl((2-b)\#\mathcal{A}_b -\#\Z_j( A)\Bigr)\\
&\ge& (2-b)(1-B)-1=1-{3\over 2}b>|\<u_j\>^d_A|,
\end{eqnarray*}
where we have used \eqref{beq} in the last estimate. We then obtain a contradiction, thus proving the claim.

\bigskip
By the Claim above, there exist $(B/2)\#\Z_j( A)$ pairs of indices $(i_s,l_s)$ with $\{i_s,l_s\}\cap\{i_r,l_r\}=\emptyset$ if $r\neq s$ and
\begin{equation}
(u_j^{i_s}, u_j^{l_s})\leq 1-b.
\end{equation}

Since $\eta_j\to0$, with fixed $c>0$ we may suppose that
\begin{align}
 \label{goodpoints}
 B\,\#\Z_j( A)>2\left\lfloor\frac{\eta_j}{c}\#\Z_j( A)\right\rfloor+1
\end{align}
for all $j$.

We now set 
\begin{equation}\label{defPj}
P_j=\Bigl\lfloor{\eta_j\over c}\#\Z_j( A)\Bigr\rfloor+1,
\end{equation}
so that by \eqref{goodpoints} there exist pairs $(i_s,l_s)$ as above, with $s\in I_j:=\{1,\ldots, P_j\}$.
Note that $P_j$ satisfies the second claim of the theorem.

If for fixed $j$ we define the vector 
$$
w=\sum_{i\in \Z_j( A)} u^i_j-\#(\Z_j( A))\<u\>_A-\sum_{s\in I_j}(u_j^{i_s}+u_j^{l_s}),
$$
then we have
\begin{eqnarray*}
|w|&\le& \#(\Z_j( A))|\<u_j\>^d-\<u\>_A|+\sum_{s\in I_j}|u_j^{i_s}+u_j^{l_s}|
\\
&\le& \#(\Z_j( A))\eta_j+\sum_{s\in I_j}\sqrt{2+2(u_j^{i_s},u_j^{l_s})}
\\
&\le& \#(\Z_j( A))\eta_j+ P_j\,\sqrt{4-2b}.
\end{eqnarray*}
Since $
{\#\Z_j( A)}\eta_j< c\,P_j 
$
by \eqref{defPj}, we then have
$
|w|\le c\,P_j + P_j\,\sqrt{4-2b}$. 
We finally choose $c>0$ such that $
\sqrt{4-2b}< 2-c$,
so that 
$$
|w|< 2P_j.
$$

By Lemma \ref{lemmapalle}, applied with $u=-w$ and $k=2P_j$, there exists a set of $2P_j$ vectors in $\S$, that we may label as $$\{\overline u_j^{i_s}, \overline u_j^{l_s}: s\in I_j\},$$
such that
\begin{equation}
\sum_{s\in I_j}(\overline u_j^{i_s}+\overline u_j^{l_s})=-w.
\end{equation}
If we now define $\tilde u_j$ by setting
\begin{equation}
\tilde u_j^i=\begin{cases}
\overline u_j^{i} & \hbox{if }i\in\{i_s,l_s: s\in I_j\}\\
u_j^{i} & \hbox{otherwise,}
\end{cases}
\end{equation}
we have
\begin{eqnarray*}
\<\tilde u_j\>^d_A&=&{1\over \#\Z_j( A)}\Bigl(\sum_{i\in \Z_j( A)} u^i_j
-\sum_{s\in I_j}(u_j^{i_s}+u_j^{l_s})+\sum_{s\in I_j}(\overline u_j^{i_s}+\overline u_j^{l_s})\Bigr)
\\
&=&{1\over \#\Z_j( A)}\Bigl(\sum_{i\in \Z_j( A)} u^i_j
-\sum_{s\in I_j}(u_j^{i_s}+u_j^{l_s})-w\Bigr)\\
&=&\<u\>_A,\end{eqnarray*}
and the proof is concluded.
\end{proof}

\begin{oss}
\label{rmmedia}
The assumption $|\<u\>_A|<1$ in Lemma \ref{lemmamedia} is sharp: if  $|\<u\>_A|=1$, we may have $u_j\rightharpoonup^* u$, such that  $u_j\neq u$ and $|\<u_j\>^{d,\delta_j}_A|=1$ at every point (for example take $u$ and $u_j$ constant vectors in $\s$). In this case, in order to have $\<u_j\>^{d,\delta_j}_A=\<u\>_A$, we should change the function $u_j$ in every point.
\end{oss}

\section{The homogenization formula}\label{homer}
In this section we prove that the homogenization formula characterizing $G_{\hom}$ in Theorem \ref{mainresult} is well defined, and derive some properties of that function.

\begin{prop}
\label{minimumexistance}
 Let $G$ be a function satisfying \eqref{fgrowth}--\eqref{fsci} and let $G^\xi$ be defined as in \eqref{gxi}. For all $T>0$ consider an arbitrary $x_T\in\R^m$, then the limit
\begin{align}\label{limber}
\lim_{T\to\infty}\frac{1}{T^m}\inf\Bigl\{E_1(u;x_T+Q_T): \
 \<u\>^{d,1}_{x_T+Q_T}=z  \Bigr\}
\end{align}
exists for all $z\in \B$.
\end{prop} 

\begin{proof}
Let $z\in\B$ be fixed. In the following we will assume $G$  to be $1$-periodic (which means that in \eqref{fperiodicity} we consider $l=1$) and $x_T=0$,  since the general case can be derived similarly following arguments already present for example in \cite{AC}  and \cite{ACG} and only needing a heavier notation. Let $t>0$ and consider the function
\begin{align}
\label{gt0}
g_t(z)=\frac{1}{t^m}\inf \Bigl\{E_1(u,\zeta;Q_t):\
  \<u\>^{d,1}_{Q_t}=z  \Bigr\}.
\end{align}
In the rest of the proof we will  drop the dependence on $z$. Let $u_t$ be a test function for $g_t$ such that
\begin{align}
\label{gt}
\frac{1}{t^m}E_1(u_t;Q_t)\leq g_t +\frac{1}{t},
\end{align}
For every $s>t$ we want to prove that $g_s<g_t$ up to a controlled error.

For fixed $s,t$, we introduce the following notation:  $$I:=\left\{0,\dots,\left\lfloor \dfrac{s}{t}\right\rfloor-1\right\}^m.$$
We can construct a test functions for $g_s$ as
\begin{align*}
&u_s(\beta)=\begin{cases}
					u_t(\beta- ti ) \qquad \text{if} \quad \beta\in  ti +Q_t \quad i\in I\\
					\bar{u}(\beta) \qquad \qquad \text{otherwise},
					\end{cases}
\end{align*}
where $\bar{u}$ is a  $\s$-valued function such that $\<u_s\>^{d,1}_{Q_s}=z$. We can choose such $\bar{u}$ thanks to Lemma \ref{lemmapalle}: define 
\begin{gather*}
\Z(Q_s)=\Z^m\cap Q_s, \qquad
 Q_{s,t}=\Bigl(\bigcup_{i\in I}( ti +Q_t)\cap\Z(Q_s)\Bigr).
\end{gather*}
We want $\bar{u}$ to be such that 
\begin{align*}
\sum_{\beta \in \Z(Q_s)}u_s(\beta)=z\#(\Z(Q_s)). 
\end{align*}
Equivalently
\begin{align*}
 \sum_{\substack{\beta\in Q_{s,t} \\  \beta \in  ti +Q_t}}u_t(\beta- ti )+\sum_{\beta\in \Z(Q_s)\backslash Q_{s,t}}\bar{u}(\beta)=z\#(\Z(Q_s)),
 \end{align*}
 which means that
 \begin{align}
\label{vectsum}
& \sum_{\beta\in \Z(Q_s)\backslash Q_{s,t}}\bar{u}(\beta)=z\Bigl(\#(\Z(Q_s))-\#(Q_{s,t})\Bigr).
\end{align}
On the left-hand side of \eqref{vectsum} we are summing $\#(\Z(Q_s))-\#(Q_{s,t})$ vectors in $\s$ while on the right-hand side we have a vector which belongs to a ball whose radius is at most $\#(\Z(Q_s))-\#(Q_{s,t})$.

If $|z|<1$, thanks to Lemma \ref{lemmapalle}  we know that it is possible to choose the values of $\bar{u}$ in such a way that the  relation \eqref{vectsum} is satisfied.

If $|z|=1$, we simply set $\bar{u}(\beta)\equiv z$, and again \eqref{vectsum} is satisfied.

Moreover we observe that
\begin{multline*}
R^\xi_1(Q_s)\subseteq \left(\bigcup_{i\in I}R^\xi_1( ti +Q_t)\right)\cup \left(R^\xi_1\Bigl(Q_s\backslash\bigcup_{i\in I}( ti +Q_t)\Bigr)\right)\cup\\
\left(\bigcup_{i\in I}( ti +(\{0,\dots,t+R\}^N\backslash\{0,\dots,t-R\}^N))\right)
\end{multline*}
and if $\beta$ belongs to one of the last two set of indices, then $D^\xi_1 \zeta_s(\beta)=M(\xi/|\xi|)$.\\
Recalling now \eqref{fgrowth}, for some $\bar{C}>0$ big enough,  we have that
\begin{align*}
&g_s\leq \frac{1}{s^m}E_1(u_s;Q_s) \\
&
\leq\left\lfloor\dfrac{s}{t} \right\rfloor^m\frac{1}{s^m}E_1(u_t;Q_t)+\frac{1}{s^m}\bar{C}\left(s^m-\left\lfloor\dfrac{s}{t} \right\rfloor^mt^m+\left\lfloor\dfrac{s}{t} \right\rfloor^m((t+R)^m-(t-R)^m)\right).
\end{align*}
Using now \eqref{gt} we get
\begin{multline}
\label{gtestim}
g_s\leq \left\lfloor\dfrac{s}{t} \right\rfloor^m \frac{t^m}{s^m}\left(g_t+\frac{1}{t}\right)+\frac{1}{s^m}\bar{C}\left(s^m-\left\lfloor\dfrac{s}{t} \right\rfloor^mt^m+\left\lfloor\dfrac{s}{t} \right\rfloor^m((t+R)^m-(t-R)^m)\right).
\end{multline}
Letting now $s\to+\infty$ and then $t\to+\infty$, we have that
\begin{align*}
\limsup_{s\to+\infty} g_s(z)\leq \liminf_{t \to+\infty}g_t(z),
\end{align*}
which concludes the proof.
\end{proof}

\begin{oss}\rm
Note that for $z\in\s$ the only test function for the minimum problem in \eqref{limber} is the constant $z$, so that the limit is actually an average over the period with $u=z$.
\end{oss}
  
  \begin{prop}
\label{ghomconvexity}  
   The function $G_{\hom}$ as defined in \eqref{fghom} is convex and lower semicontinuous in $\B$.
    \end{prop}
  \begin{proof}
We want to show that  for every $ 0\leq t\leq 1$  and for every $z_1,z_2\in \B$ it holds:
  \begin{align}
  G_{\hom}(tz_1+(1-t)z_2,M)\leq tG_{\hom}(z_1,M)+(1-t)G_{\hom}(z_2,M). 
  \end{align}
Let $k\in \N$ be fixed; having in mind \eqref{fperiodicity} and thanks to Proposition \ref{minimumexistance}, it is nor restrictive to take $k\in l\N$. We define 
\begin{align}
\label{gk}
g_k(z)=\frac{1}{k^m}\inf \Bigl\{E_1(u;Q_k):\
\<u\>^{d,1}_{Q_k}=z \Bigr\}.
\end{align}
 In the following we will denote $g_k^1=g_k(z_1)$, $g_k^2=g_k(z_2).$ 
 
Let $u_k^1$ and $u_k^2$ be functions such that
\begin{align}
\label{gk1}
&\frac{1}{k^m}E_1(u_k^1;Q_k)\leq g_k^1 +\frac{1}{k},\\
\label{gk2}
&\frac{1}{k^m}E_1(u_k^2;Q_k)\leq g_k^2 +\frac{1}{k}.
\end{align}
Let $h>k$ be such that $h/k\in\N$.
Denote $g_h=g_h(tz_1+(1-t)z_2)$, we define the following test function for $g_h$:
\begin{align*}
&u_h(\beta)=\\
&\begin{cases}
					u_k^1(\beta-ki) \quad \text{if} \quad \beta\in ki+Q_k \quad i\in\left\{0,\dots,\dfrac{h}{k}-1\right\}^{m-1}\times\left\{0,\dots,\left\lfloor \dfrac{h}{k}t\right\rfloor-1\right\} \\
					u_k^2(\beta-ki) \quad \text{if} \quad \beta\in ki+Q_k \quad i\in\left\{0,\dots, \dfrac{h}{k}-1\right\}^{m-1}\times\left\{\dfrac{h}{k}-\left\lfloor \dfrac{h(1-t)}{k}\right\rfloor,\dots,\dfrac{h}{k}-1\right\} \\
					\bar{u}(\beta) \quad \qquad \text{otherwise},
					\end{cases}
\end{align*}
Reasoning as in  Proposition \ref{minimumexistance}, thanks to Lemma \ref{lemmapalle}, we can choose the values of $\bar{u}$ such that $\<u_s\>^{d,1}_{Q_s}=tz_1+(1-t)z_2$.\\
By \eqref{fgrowth} and \eqref{fperiodicity}, for some $\bar{C}>0$ we get
\begin{align*}
g_h&\leq \frac{1}{h^m}E_1(u_h,\zeta_h;Q_h)\\
&\leq\frac{1}{h^m}\left(\frac{h}{k}\right)^{m-1}\left\lfloor\dfrac{h}{k}t\right\rfloor E_1(u_k^1;Q_k)
+\frac{1}{h^m}\left(\frac{h}{k}\right)^{m-1}\left\lfloor\dfrac{h(1-t)}{k}\right\rfloor E_1(u_k^2;Q_k)\\
&+\frac{1}{h^m}\bar{C}\left(h^m-\left(\frac{h}{k}\right)^{m-1}\left(\left\lfloor\dfrac{h}{k}t\right\rfloor +\left\lfloor\dfrac{h(1-t)}{k}\right\rfloor\right) k^m\right)\\
&+\frac{1}{h^m}\bar{C}\left(\frac{h}{k}\right)^m((k+R)^m-(k-R)^m).
\end{align*}
Then, thanks to \eqref{gk1} and \eqref{gk2}, we can rewrite the above relation as 
\begin{align*}
g_h&\leq \frac{k^m}{h^m}\left(\frac{h}{k}\right)^{m-1}\left\lfloor\dfrac{h}{k}t\right\rfloor\left(g_k^1+\frac{1}{k}\right)+\frac{k^m}{h^m}\left(\frac{h}{k}\right)^{m-1}\left\lfloor\dfrac{h(1-t)}{k}\right\rfloor\left(g_k^2+\frac{1}{k}\right)\\
&+\frac{1}{h^m}\bar{C}\left(h^m-\left(\frac{h}{k}\right)^{m-1}\left(\left\lfloor\dfrac{h}{k}t\right\rfloor +\left\lfloor\dfrac{h(1-t)}{k}\right\rfloor\right) k^m\right)\\
&+\frac{1}{h^m}\bar{C}\left(\frac{h}{k}\right)^m((k+R)^m-(k-R)^m).
\end{align*}
Letting $h\to+\infty$ and then $k\to+\infty$, we can conclude the proof of the convexity.

From the convexity and the boundedness of $G_{\hom}$ we deduce that it is continuous in the interior of $\B$. 
Moreover, by \eqref{fsci} it is lower semicontinuous at points on the boundary of $\B$
  \end{proof}

  \begin{oss}\rm
  Note that the function $G_{\hom}$ may not be continuous on $\B$, even if it is convex. It suffices to take a nearest-neighbor energy $G=G^\xi$ independent on $\alpha$, with $G(e_1, e_1)=0$ and $G(u,v)= 1$ otherwise. In this case, in particular $G_{\hom}(e_1)=0$ and $G_{\hom}(z)=1$ if $z\in \s$, $z\neq e_1$.
  \end{oss}

\section{Proof of the homogenization theorem}\label{homproof}
Thanks to the geometric lemmas in Section \ref{lemmas}, we can now easily give a proof of the homogenization theorem. 
We remark that it will be sufficient to prove a lower bound, since we may resort to the homogenization result of Alicandro, Cicalese and Gloria \cite{ACG} in order to give an upper bound for the homogenized functional. Indeed, by \eqref{ABGhom} we have $\overline G_{\hom}\le  G_{\hom}$, so that functional \eqref{home} is an upper bound for the homogenized energy. Note that we could directly proof the upper bound using approximation results and constructions starting from the formula for $G_{\hom}$, but this would be essentially a repetition of the arguments in \cite{ACG}.

\bigskip
In order to prove a lower bound we will make use of Lemma \ref{lemmamedia} and of the Fonseca-M\"uller blow-up technique \cite{BRA08,FON}  .
Let $\e_j\to 0$ be a vanishing sequence of  parameters, let  $u\in L^{\infty}(\Omega,\B)$ and let $u_j\rightharpoonup^* u$ with $u\in L^{\infty}(\Omega,\B)$. 
We define the measures $\mu_j$ by setting
\begin{gather*}
\mu_j(A)=\sum_{\substack{\xi\in \Z^m \\ |\xi|\leq R  }}\sum_{\alpha\in R_{\e_j}^{\xi}(\Omega)}\e_j^mG^\xi\left(\frac{\alpha}{\e_j}, u_j(\alpha),u_j(\alpha+\e_j\xi)\right)\delta_{\alpha+\frac{\e_j}{2}\xi}(A)
\end{gather*}  for all $A\in\mathcal{B}(\R^n)$,
where $\delta_x$  denotes the usual Dirac delta at $x$. Since the measures are equibounded, by the weak* compactness of measures there exists a limit measure $\mu$ on $\Omega$ such that, up to subsequences, $\mu_j\rightharpoonup^* \mu$.
We consider the Radon-Nikodym decomposition of the limit measure $\mu$ with respect to the $m$-dimensional Lebesgue measure $\mathcal{L}^m$:
\begin{gather*}
\mu=\frac{d\mu}{dx}d\mathcal{L}^m+\mu^s, \qquad \mu^s\bot\,\mathcal{L}^m.
\end{gather*}
Besicovitch Derivation Theorem \cite{uno} states that almost every point in $\Omega$ with respect to $\mathcal{L}^m$ is a Lebesgue point for $\mu$. So, we may suppose that $x_0\in\Z_{\e_j}(\Omega)$ be a Lebesgue point both for $u$ and for $\mu$ and let $Q_\rho(x_0)=x_0+(-\rho/2,\rho/2)^m$. We then have
\begin{align}
\label{lebesguepoint}
\frac{d\mu}{dx}(x_0)=\lim_{\rho\to 0^+}\frac{\mu(Q_\rho(x_0))}{\mathcal{L}^m(Q_\rho(x_0))}=\lim_{\rho\to 0^+}\frac{1}{\rho^m}\mu(Q_\rho(x_0)).
\end{align} 
Recalling that
\begin{gather}
\label{sequence} \mu(Q_\rho(x_0))=\lim_{j\to+\infty}\mu_j(Q_\rho(x_0))
\end{gather}
except for a countable set of $\rho$,
by a diagonalization argument on \eqref{lebesguepoint} and \eqref{sequence} we can extract a subsequence $\{\rho_j\}$ such that it holds
\begin{align*}
\frac{d\mu}{dx}(x_0)=\lim_{j\to +\infty}\frac{1}{\rho_j^m}\mu_j(Q_{\rho_j}(x_0)).
\end{align*}
This means that
\begin{align}
\label{measure1}
\frac{d\mu}{dx}(x_0)=
\lim_{j\to +\infty}\left(\frac{\e_j}{\rho_j}\right)^m\sum_{\substack{\xi\in \Z^m \\ |\xi|\leq R  }}\sum_{\alpha\in R_{\e_j}^{\xi}(\Omega)}G^\xi\left(\frac{\alpha}{\e_j}, u_j(\alpha),u_j(\alpha+\e_j\xi)\right)\delta_{\alpha+\frac{\e_j}{2}\xi}(Q_{\rho_j}(x_0)).
\end{align}

Also note that, by the weak$^*$ convergence of $u_j$ to $u$, we have
$ \<u\>_{Q_{\rho}(x_0)}=\lim_j  \<u_j\>^{\e_j}_{Q_{\rho}(x_0)}=\lim_j  \<u_j\>^{d,\e_j}_{Q_{\rho}(x_0)}$,
and that for almost every $x_0$ we have
$\displaystyle \lim_{\rho\to 0^+} \<u\>_{Q_{\rho}(x_0)}= u(x_0)$,
so that we may assume that
$$
\lim_j  \<u_j\>^{d,\e_j}_{Q_{\rho_j}(x_0)}=u(x_0).
$$
We can parameterizing functions on a common unit cube, by setting
$
\delta_j={\e_j\over\rho_j}$,and $ v_j(\gamma)=u_j(x_0+\rho_j\gamma).
$
With this parameterization,  \eqref{measure1} reads
$$
\frac{d\mu}{dx}(x_0)=
\lim_{j\to +\infty}\delta_j^m\sum_{\substack{\xi\in \Z^m \\ |\xi|\leq R  }}\sum_{\gamma+{x_0\over\rho_j}\in R_{\delta_j}^{\xi}({1\over\rho_j}\Omega)}G^\xi\left(\frac{x_0}{\e_j}+{\gamma\over\delta_j}, v_j(\gamma),v_j(\gamma+\delta_j\xi)\right)\delta_{\gamma-{x_0\over\rho_j}+\frac{\delta_j}{2}\xi}(Q_{1}(0)).
$$

Note that we have
$\lim_j  \<v_j\>^{d,\delta_j}_{Q_{1}(0)}=u(x_0)$.
We can then apply Lemma \ref{lemmamedia} with $v_j$ in the place of $u_j$ and $A=Q_1(0)$. We obtain a family $\widetilde v_j$ with 
\begin{equation}\label{limage2}
\<\widetilde v_j\>^{d,\delta_j}_{Q_{1}(0)}=u(x_0),
\end{equation}
and such that $\widetilde v_j(\gamma)= v_j(\gamma)$ except for a set $P_j$ of $\gamma$ with $\#P_j=o(\delta_j^{-m})$.
From this, the finiteness of the range of interactions and the boundedness of $G^\xi$, we further rewrite  as
$$
\frac{d\mu}{dx}(x_0)=
\lim_{j\to +\infty}\delta_j^m\sum_{\substack{\xi\in \Z^m \\ |\xi|\leq R  }}\sum_{\gamma+{x_0\over\rho_j}\in R_{\delta_j}^{\xi}({1\over\rho_j}\Omega)}G^\xi\left(\frac{x_0}{\e_j}+{\gamma\over\delta_j}, \widetilde v_ j(\gamma),\widetilde v_j(\gamma+\delta_j\xi)\right)\delta_{\gamma-{x_0\over\rho_j}+\frac{\delta_j}{2}\xi}(Q_{1}(0)).
$$

We now set $$T_j=\frac{\rho_j}{\e_j}=\delta_j^{-1}, \qquad x_{T_j}={x_0\over\e_j},$$
So that
\begin{eqnarray*}
&&\frac{d\mu}{dx}(x_0)\\
&=&\lim_{j\to +\infty}{1\over T_j^m}\sum_{\substack{\xi\in \Z^m \\ |\xi|\leq R  }}\sum_{\gamma+{x_{T_j}\over T_j}\in R_{T_j^{-1}}^{\xi}({1\over\rho_j}\Omega)}G^\xi\left(x_{T_j}+{T_j\gamma}, \widetilde v_ j(\gamma),\widetilde v_j(\gamma+{\xi\over T_j})\right)\delta_{\gamma-{x_{T_j}\over T_j}+\frac{1}{2T_j}\xi}(Q_{1}(0))
\\
&=&\lim_{j\to +\infty}{1\over T_j^m}\sum_{\substack{\xi\in \Z^m \\ |\xi|\leq R  }}\sum_{\eta+{x_{T_j}}\in R_{1}^{\xi}({1\over\e_j}\Omega)}G^\xi\left(x_{T_j}+{\eta}, \widetilde v_ j({\eta\over T_j}),\widetilde v_j({\eta+{\xi}\over T_j})\right)\delta_{\eta-{x_{T_j}}+\frac{1}{2}\xi}(Q_{T_j}(0)).
\end{eqnarray*}

We now set $w_j(\eta)=  v_ j({\eta\over T_j})$ and use the boundedness of $G^\xi$ and $R$ to deduce that
\begin{eqnarray*}
&&\frac{d\mu}{dx}(x_0)\\
&\ge&\lim_{j\to +\infty}{1\over T_j^m}\sum_{\substack{\xi\in \Z^m \\ |\xi|\leq R  }}\sum_{\eta+{x_{T_j}}\in R_{1}^{\xi}(Q_{T_j})}G^\xi\left(x_{T_j}+{\eta}, \widetilde w_ j({\eta}),\widetilde v_j({\eta+{\xi}})\right)\delta_{\eta-{x_{T_j}}+\frac{1}{2}\xi}(Q_{T_j}(0))
\end{eqnarray*}
(note indeed that by considering interactions in $R_{1}^{\xi}(Q_{T_j})$ we neglect a contribution of a number of interactions of order $O(T_j^{m-1})$; i.e., an energy contribution of order $O(T_j^{-1})$). Noting that $w_j$ satisfies the constraint
$$\<\widetilde w_j\>^{d,1}_{{x_{T_j}+Q_{T_j}(0)}}=u(x_0),$$
thanks to \eqref{limage2}, by Proposition \ref{minimumexistance} we finally deduce that 
\begin{eqnarray*}
&&\frac{d\mu}{dx}(x_0)\ge \lim_{j\to\infty}\frac{1}{T_j^m}\inf\Bigl\{E_1(w;x_{T_j}+Q_{T_j}): \
 \<w\>^{d,1}_{x_{T_j}+Q_{T_j}}=u(x_0)  \Bigr\}= G_{\hom}(u(x_0)).
\end{eqnarray*}
Since this holds for almost all $x_0\in\Omega$, we have proved the desired lower bound.

\subsection*{Acknowledgments.}
The authors acknowledge the MIUR Excellence Department Project awarded to the Department of Mathematics, University of Rome Tor Vergata, CUP E83C18000100006.


\begin{thebibliography}{20}
  \addcontentsline{toc}{chapter}{Bibliography}
 \bibitem{AC}Alicandro R., Cicalese M., A general integral representation result for continuum limits of discrete energies with superlinear growth, \textit{SIAM J. Math. Anal.} \textbf{36} 1-37 (2004)
 
\bibitem{XY} Alicandro, R., Cicalese, M. Variational analysis of the asymptotics of the XY Model. {\em Arch Rational Mech Anal} {\bf 192}, 501--536 (2006)
 
 \bibitem {ACG} Alicandro R., Cicalese M.,  Gloria A., Variational description of bulk energies for bounded and unbounded spin systems, \textit{Nonlinearity} \textbf{21} 1881-1910  (2008)

\bibitem{uno} Ambrosio L., Fusco N.,  Pallara D.,  Function of Bounded Variations and Free Discontinuity Problems, \textit{Oxford University Press} Oxford (2000)

 \bibitem{ZAN} Berardi R., Emerson A.P.J., Zannoni C., Monte Carlo Investigations of a Gay-Berne Liquid Crystal, 
\textit{J. Chem. Soc. Faraday Trans.} \textbf{89} (22) 4069-4078 (1993) 

 \bibitem{GAYB} Berne B.J., Gay J.G., Modification of the overlap potential to mimic linear site-site potential, \textit{J. Chem. Phys.} \textbf{74} (6) (1981)

\bibitem{BRA1}Braides A., $\Gamma$-convergence for Beginners, \textit{Oxford University Press} (2002)

 \bibitem{BH} Braides A., A Handbook of $\Gamma\hbox{-}$convergence In \emph{Handbook of Differential Equations. Stationary Partial Differential Equations} (M. Chipot and P. Quittner, eds.) \emph{Elsevier} \textbf{3} (2006)

\bibitem{BRACIC} Braides A.,  Cicalese M., Solombrino F., Q-tensor continuum energies as limits of head-to-tail symmetric spin systems, \textit{SIAM J. Math. Anal.} \textbf{47}  2832-2867 (2015)

\bibitem{BLO}Braides A., Lew A.J., Ortiz M.,
Effective cohesive behavior of layers of interatomic planes,
 \textit{Arch. Ration. Mech. Anal.}  \textbf{180} 151-182 (2006)

\bibitem{BRA08} Braides A., Maslennikov M., Sigalotti L., Homogenization by blow-up, \textit{ Applicable Analysis: An International Journal} \textbf{ 87} (12) 1341-1356 (2008)

\bibitem{BOx} Braides A., Solci M., Asymptotic analysis of Lennard-Jones systems beyond the nearest-neighbour setting: a one-dimensional prototypical case, \emph{Math. Mech. Solids} {\bf 21} 915-930 (2016)

\bibitem{BT} Braides  A., Truskinovsky L., Asymptotic expansions by $\Gamma$-convergence, \emph{Cont. Mech. Therm.} {\bf 20} 21- 62 (2008)

\bibitem{CO} M. Cicalese M., Orlando G., Ruf M.,
From the N-clock model to the XY model: emergence of concentration effects in the variational analysis.
Preprint, 2019.

\bibitem{CRS} Cicalese M.,  Ruf M., Solombrino F., Chirality transitions in frustrated $S^2\hbox{-}$valued spin systems, \emph{M3AS} {\bf 26}  1481-1529. (2016)

\bibitem{CS} Cicalese M., Solombrino F., Frustrated Ferromagnetic Spin Chains: A Variational
Approach to Chirality Transitions, \emph{J. Nonlinear Sci.} {\bf 25} 291-313 (2015)

\bibitem{deG}de Gennes P.G., The Physics of Liquid Crystals, \textit{   Clarendon Press } Oxford (1974)

\bibitem{FON}  Fonseca I.,  M\"uller S.,  Quasiconvex integrands and lower semicontinuity in $L^1$, \textit {SIAM J. Math. Anal} \textbf{23} 1081-1098 (1992)

\bibitem{Vir}  Virga E.G., Variational Theories for Liquid Crystals, \textit{Chapman and Hall}, London (1994)

\bibitem{MOLMOD} Wu J., Variational Methods in Molecular Modeling, \textit {Springer} (2017)

 \end{thebibliography}
\end{document}